\documentclass[11pt,reqno]{amsart}

\theoremstyle{definition}

\numberwithin{equation}{section}

\newcommand{\be}{\beta}

\theoremstyle{definition}\newtheorem{thm}{Theorem}[section]
\theoremstyle{definition}\newtheorem{cor}[thm]{Corollary}
\theoremstyle{definition}\newtheorem{lem}[thm]{Lemma}
\theoremstyle{definition}\newtheorem{prop}[thm]{Proposition}
\theoremstyle{definition}\newtheorem{defn}[thm]{Definition}
\theoremstyle{definition}\newtheorem{Rem}[thm]{Remark}
\theoremstyle{definition}\newtheorem{exam}[thm]{Example}
\theoremstyle{definition}

\def\be{\begin{enumerate}}
	\def\ee{\end{enumerate}}
\usepackage{amssymb}

\begin{document}
\title{Countable strongly annihilated ideals in commutative rings}
\author{R. Mohamadian} 
\address{Department of Mathematics, Shahid Chamran University of Ahvaz, Ahvaz,
	Iran}
\email{mohamadian\_r@scu.ac.ir}
\maketitle
\begin{abstract}
In this paper we introduce and study the concept of  countable strongly annihilated ideal in commutative rings, in particular in rings of continuous functions. We show that a maximal ideal in $C(X)$ is  countable strongly annihilated if and only if it is a real maximal  $z^\circ$-ideal. It turns out that $X$ is an almost $P$-space if and only if countable strongly annihilated ideals and strongly divisible $z$-ideals coincide.  We observe that an almost $P$-space $X$ is Lindel\"{o}f if and only if every countable strongly annihilated ideal is fixed. We give a negative answer to a question raised by Gilmer and McAdam.
\end{abstract}

\vspace{3mm}\noindent\textbf{Keywords:}  countable strongly annihilated ideal, real maximal ideal, strongly divisible ideal, almost $P$-space.  

\vspace{2mm}\noindent\textbf{MSC:} Primary 13A18; Secondary 54C40.

\section{Introduction}
 Throughout this paper all rings are commutative with unity. Let $R$ be a ring and $a\in R$. The principal ideal generated by $a$ denotes by $(a)$. By ${\rm rad}(R)$ we mean the prime radical of $R$. Clearly, ${\rm rad}(R)$ is the set of all nilpotent elements of $R$. If ${\rm rad}(R)=(0)$ then $R$ is called a reduced ring. Elements of a ring $R$ that are not zerodivisors are called regular, and the set
 of all regular elements of $R$ is denote by $r(R)$. An ideal of $R$ is called regular if it  intersects $r(R)$; otherwise, it is called nonregular. The classical ring of quotients of  $R$ denoted by $q(R)$.  
 For each $S\subseteq R$ let $P_S$ (resp. $M_S$) be the intersection of all minimal prime (resp. maximal) ideals of $R$  containing $S$. We use $P_a$ (resp. $M_a$) instead of $P_{\{a\}}$ (resp. $M_{\{a\}}$). A proper ideal $I$ of a ring $R$ is called a $z^\circ$-ideal (resp. $z$-ideal) if for  each $a\in I$, we have $P_a\subseteq I$ (resp. $M_a\subseteq I$).  A proper ideal $I$ of a ring $R$ is called a $sz^\circ$-ideal (resp. $sz$-ideal) if for each $S\subseteq I$, we have $P_S\subseteq I$ (resp. $M_S\subseteq I$).
 If $I$ is an ideal of $R$, then ${\rm Min}(I)$ denotes the set of all prime ideals minimal over $I$ and if $a\in R$, then $I(a)$ denotes $a+I$ in ring of fraction $\frac{R}{I}$.
 
 All topological spaces are completely regular Hausdorff. If $X$ is a space and $A\subseteq X$ then $A^\circ$ and $\overline{A}$ denote ${\rm int}(A)$ and ${\rm cl}(A)$ respectively. $C(X)$ is the ring of all continuous real valued functions on the space $X$ and $C^*(X)$ is the subring of $C(X)$ consisting of all bounded functions in $C(X)$.  The space $X$ is called pseudocompact whenever $C^*(X)=C(X)$. For $f\in C(X)$, the zero-set of $f$ is the set $Z(f)=\{x\in X:f(x)=0\}$. The set-theoretic complement of $Z(f)$ is denoted by $coz(f)$. The support of a function $f\in C(X)$ is the $\overline{coz(f)}$. It is well-known that $C(X)$ is a reduced ring. A topological space $X$ is said to be a $P$-space if $Z(f)=Z^\circ(f)$ for every $f\in C(X)$. The space $X$ is called an almost $P$-space if $Z(f)\neq\emptyset$ implies that $Z^\circ(f)\neq\emptyset$, for every $f\in C(X)$. It is well-known that an ideal $I$ of $C(X)$ is $z^\circ$-ideal (resp. $z$-ideal) if $Z^\circ(f)=Z^\circ(g)$ (resp. $Z(f)=Z(g)$) $f\in I$ and $g\in C(X)$ implies that $g\in I$.  For an ideal $I$ of $C(X)$, we write $Z[I]$ to designate the family of zero-sets $\{Z(f):f\in I\}$. $\upsilon X$ is the Hewitt realcompactification of $X$, $\beta X$ is the Stone–$\check{\rm C}$ech compactification of $X$ and for any $p\in\beta X$, the maximal ideal $M^p$ (resp. the ideal $O^p$) is the set of all $f\in C(X)$ for which $p\in {\rm cl}_{\beta X}Z(f)$ (resp. $p\in{\rm int}_{\beta X}{\rm cl}_{\beta X}Z(f)$). For $p\in X$, we set $M_p = \{f\in C(X):p\in Z(f)\}$ and $O_p = \{f\in C(X) : p\in Z^\circ(f)\}$. Clearly $M_p$ is a maximal ideal and hence it
 is a $z$-ideal and also $O_p$ is a $z^\circ$-ideal.
 A maximal ideal $M$ of $C(X)$ is called real if $\frac{C(X)}{M}\cong\Bbb R$ and it is well-known that $\upsilon X=\{p\in \beta X: M^p ~\mbox{is a real maximal ideal}\}$. A space $X$ is called realcompact if $\upsilon X=X$. 
 For more information about commutative rings, see \cite{AM3}, about general topology, see \cite{E}, about rings of continuous functions, see \cite{GJ} and about $z$-ideals and $z^\circ$-ideals in $C(X)$ and commutative rings, see \cite{M, AKR, AKR0, AM1, AM2, AM4, AM5}.
 
 The paper is organized as follows.  In Section 2, we introduce the concept of the countable strongly annihilated ideal in commutative rings and we investigate their relation to the strongly annihilated and weakly annihilated ideals. In Section 3,  we will answer to a question raised by Gilmer and McAdam in \cite{GM}. In Section 4, we investigate the behavior of the countable strongly annihilated ideals in rings of continuous functions. It turns out that a maximal ideal  $M$ of $C(X)$ is countable strongly annihilated if and only if it is a real maximal $z^\circ$-ideal.

 \section{Countable strongly annihilated ideals}

We start by recalling definition from \cite{GM}. An ideal $I$ of a ring $R$ is called  weakly annihilated (resp. strongly annihilated) if for each $a\in I$ (resp. finite subset $\{a_1,\ldots,a_n\}$ of $I$) and $b\in R\setminus I$, there exists $c\in R$ such that $ca=0$ (resp. $ca_i = 0$ for $i= 1,\ldots, n$,) but $cb\neq 0$.\\
In Lemma 2.15 of \cite {G}, it is shown that in a reduced ring weakly annihilated ideals and $z^\circ$-ideals coincide and in Lemma 2.16 of the same reference it is also shown that  strongly annihilated ideals and $sz^\circ$-ideals  coincide. For details about the aforementioned concepts see \cite{GM, HL, G}.  
\begin{defn}
 An ideal $I$ of a ring $R$ is called countable strongly annihilated if for each countable subset  $\{a_1,a_2,\ldots\}$ of $I$ and $b\in R\setminus I$, there exists $c\in R$ such that $ca_n=0$, for $n=1,2,\ldots$, but $cb\neq 0$.
\end{defn}
Every countable strongly annihilated is strongly annihilated but not conversely, see Example \ref{e1}.  Every strongly annihilated is weakly annihilated, but not conversely, see Example 4.2 in \cite{AM1}. For any ideal $I$ of $R$, ${\rm Ann}(I)$ is countable strongly annihilated. A minimal ideal of a ring is weakly annihilated if and only if countable strongly annihilated. Every element of a weakly annihilated ideal is zerodivisor.

\begin{Rem}
Every reduced principal ideal ring  $R$ is a von Neumann regular ring if and only if every its ideal is countable strongly annihilated. To see this, first suppose that $R$ is a  von Neumann regular ring and $I$ be an ideal of $R$. There is $a\in R$ such that $I=(a)$. In fact, there exists an idempotent $e\in R$ such that $I={\rm Ann}(e)$, that shows $I$ is countable strongly annihilated. The converse is clear by Corollary 1.14 of \cite{AKR}. 
\end{Rem}
\begin{Rem}
It is clear that the only weakly annihilated ideal of an integral domain is zero ideal, but the converse is false. For example, we consider the ring $\Bbb Z_4$.  
\end{Rem}
In  Question 5 of \cite{HL}, Heinzer and Lantz raised the question:  Is a maximal weakly annihilated ideal prime? It seems that this question is settled. However, here we give an affirmative answer to this question. Let us state the following easy, yet crucial, lemma.

\begin{lem}\label{L1}
Let $I$ be a  weakly annihilated ideal of a ring $R$. Then $(I:r)$ is a  weakly annihilated ideal for each $r\notin I$.
\end{lem}		
\begin{proof}
Assume that $a\in (I:r)$ and $b\notin (I:r)$. Hence $ar\in I$ and $br\notin I$. By hypothesis,  there exists $c\in R$ such that $car=0$ but $cbr\neq 0$. Take $d=cr$. This yields $da=0$ and $db\neq 0$. Thus $(I:r)$ is a weakly annihilated ideal.
\end{proof}	

\begin{prop}
Every maximal weakly annihilated ideal is prime.
\end{prop}	 	

\begin{proof}
Let $I$ be a  maximal weakly annihilated ideal of a ring $R$. We must show that $I$ is prime. Suppose that $ab\in I$ and $a\notin I$. Clearly, $I\subseteq (I:a)\neq R$ and $b\in (I:a)$. Now by maximality of $I$ and by Lemma \ref{L1} we conclude that $I=(I:a)$ and we are done.   
\end{proof} 

Similar to the proof of the above lemma and proposition we can show that if $I$ is a countable strongly annihilated ideal then $(I:r)$ is a  countable strongly annihilated ideal for each $r\notin I$ and every maximal countable strongly annihilated ideal is prime.

Let $M$ be a maximal ideal of a reduced ring $R$ and $R(x)$ means the localization of $R[x]$ at the prime ideal ${\mathcal P}=M[x]$. If $I$ is a $z^\circ$-ideal of $R$ which contained in $M$,  then $(I[x])_{\mathcal P}$ may not be a $z^\circ$-ideal in $R(x)$. See Example \ref{e44}. First we need the following lemma.

\begin{lem}\label{L2}
Let $Q, P$ be two  prime ideals in $R$,  $Q\subseteq P$ and $R_P$ be the localization $R$ at $P$. If $Q_P$ is a weakly annihilated ideal in $R_P$, then $Q$ is a weakly annihilated ideal in $R$. 
\end{lem}

\begin{proof}
Suppose that  $a\in Q$ and $b\in R\setminus Q$. Then $\frac{a}{1}\in Q_P$ and $\frac{b}{1}\in R_P\setminus Q_P$. By hypothesis, there is $\frac{r}{s}\in R_P$ such that $\frac{ra}{s}=\frac{0}{1}$ and $\frac{rb}{s}\neq\frac{0}{1}$. Hence, there is $t_0\notin P$ such that $rat_0=0$ and also $rbt\neq 0$, for every $t\notin P$. Take $d=rt_0$, then $da=0$ but $db\neq 0$ and we are done. 
\end{proof}
It is well-known that if $R$ is a reduced ring then $R[x]$ is reduced and if $P$ is a prime ideal of $R$, then $R_P$ is also a reduced ring.

\begin{exam}\label{e44}
Suppose that $R$ is the same as the ring in Example 4.2 in \cite{AM1}. It is shown that the ideal $I$ is contained in a $z^\circ$-ideal, while there is no $sz^\circ$-ideal containing $I$. Suppose $\mathcal A$ be the set of all $z^\circ$-ideals of $R$ containing $I$, which is not $sz^\circ$-ideal. By Zorn's Lemma  $\mathcal A$ has a maximal element $J$. Now it can be easily seen that $J$ is a prime $z^\circ$-ideal which is not a $sz^\circ$-ideal.  Since $J$ is not a $sz^\circ$-ideal, then by Proposition 3.9 of \cite{AM1}, we infer that $J[x]$ is not a $z^\circ$-ideal in $R[x]$. Now by the above lemma $(J[x])_{M[x]}$ is not a $z^\circ$-ideal in $R(x)$, where $M$ is the maximal ideal of $R$ which containing $J$. 	
\end{exam}

For $f\in R[x]$, let $C(f)$ denotes the set of all coefficients of $f$.
\begin{prop}
The following statements are equivalent.\\	
a) $I$ is a	countable strongly annihilated ideal in $R$.\\
b) $I[x]$ is a countable strongly annihilated ideal in $R[x]$.\\
c) $I[[x]]$ is a	countable strongly annihilated ideal in $R[[x]]$.
	
\end{prop}
\begin{proof}
$(a\Leftrightarrow b)$ Let $S=\{f_n:n=1,\ldots\}$ be a countable subset of $I[x]$ and $g\notin I[x]$. Hence $C=\bigcup_{n=1}^\infty C(f_n)$ is a countable subset of $I$ and there is also an element  ${i_0}$ such that $g_{i_0}\notin I$. By hypothesis, there exists $r\in R$ such that $rC=0$ and $rg_{i_0}\neq0$. This means that $rf_n=0$, for every $n$, and $rg\neq 0$. The converse is straightforward. 
 
$(a\Leftrightarrow c)$ Similar to implication $(a\Leftrightarrow b)$.
\end{proof} 	
\begin{prop}
The following statements are hold.\\
a)  If $I$ is a	countable strongly annihilated ideal in $R$, then $I_S$ is a	countable strongly annihilated   ideal in $q(R)$.\\
b) Let $P$ is  a nonregular ideal in $R$. If $P_S$ is a	countable strongly annihilated ideal in $q(R)$, then $P$ is a countable strongly annihilated ideal in $R$.

\end{prop}
\begin{proof}
It is straightforward.
\end{proof}

An ideal $I$ of a ring $R$ is called $\lambda-z^\circ$-ideal whenever $P_S\subseteq I$, for every $S\subseteq I$ with $|S|\leq\lambda$, where $\lambda$ is a cardinal number, see \cite{AM2}. It is obvious that every minimal prime ideal is  a $\lambda-z^\circ$-ideal, where $\lambda\leq|P|$.  Refer to Example \ref{e11} to see a maximal ideal that is a  $c-z^\circ$-ideal

\begin{prop}
Every countable strongly annihilated ideal $I$ of  a reduced ring $R$ is an $\aleph_0-z^\circ$-ideal.	
\end{prop}
\begin{proof}
Let  $S=\{a_n:n=1,\ldots\}$ be a countable subset of $I$ and $b\in P_S$. If $b\notin I$, then there is $c\in R$ such that $ca_n=0$, for any $n$, and $cb\neq0$. Since $cb\notin{\rm rad}(R)=(0)$, there is a minimal prime ideal $Q$ such that $cb\notin Q$. Now $cS=0\in Q$ and $c\notin Q$ conclude that $S\subseteq Q$. This is consequence that $b\in Q$ which is a contradiction. 
\end{proof}

The converse of the above proposition is false. See Example \ref{e1}.

Recall that a ring $R$  satisfies property $A$ if each finitely generated nonregular ideal has a nonzero annihilator. A ring $R$ is said to have 
the annihilator condition (briefly a.c.) if for each finitely generated ideal $I$ of $R$ there exists an element $b\in R$ with ${\rm Ann}(I)={\rm Ann}(b)$. If this element $b\in R$ can be chosen in $I$, then we say $R$ satisfies the strong  annihilator condition (briefly s.a.c.). For details, see \cite{H, L, AM1, HJ}.

\begin{defn}
We say that an ideal $I$ of a ring $R$ satisfies the countable strongly annihilator condition (briefly c.s.a.c.) if for any countable subset $S$ of $I$, there exists an element $a\in I$ such that ${\rm Ann}(S)={\rm Ann}(a)$. 	
\end{defn}

\begin{prop}
Let $I$ be a weakly annihilated ideal satisfies the c.s.a.c., then $I$ is countable strongly annihilated.
\end{prop}
\begin{proof}
Let  $S=\{a_n:n=1,\ldots\}$ be a countable subset of $I$ and $b\notin I$. There is $c\in I$ such that ${\rm Ann}(S)={\rm Ann}(c)$. By hypothesis,  $cd=0$ and $db\neq0$, for some $d\in R$. This implies that $dS=0$ and we are through. 
\end{proof}

The converse is not true, in general. See the next example.
\begin{exam}\label{e22}
Let $R=\frac{\Bbb Z_2[x,y]}{I}$,  $I=(x^2,y^2)$ and suppose that $\mathcal J=\frac{(x,y)}{(x^2,y^2)}$. In Example 3.13 of \cite{L} it is shown that $\mathcal J$ does not have c.s.a.c.. It is not hard to see that $I(xy)I(f)=0$, for every $I(f)\in \mathcal J$ and $I(xy)I(g)\neq0$, for every $I(g)\notin \mathcal J$. This shows that $\mathcal J$ is countable strongly annihilated. 
\end{exam}

\begin{prop}\label{p22}
Let $I$ be an ideal  of a Noetherian ring $R$. Then\\
a)  $I$ satisfies the c.s.a.c. if and only if satisfies the s.a.c.\\
b) $I$ is a strongly annihilated ideal if and only if countable strongly annihilated.

\end{prop}
\begin{proof}
$(a\Rightarrow )$	It is obvious.\\
$(a\Leftarrow )$
Let  $S=\{a_n:n=1,\ldots\}$ be a countable subset of $I$. There exists $n\in \Bbb N$ such that $(a_1,\ldots,a_n)=(a_1,\ldots,a_{n+1})=\cdots$. Now by hypothesis there is $c\in I$ such that ${\rm Ann}(\{a_1,\ldots,a_n\})={\rm Ann}(c)$. Since $a_{n+1}=r_1a_1+\cdots+r_na_n$, where $r_1,\ldots,r_n\in R$, we infer that ${\rm Ann}(S)={\rm Ann}(c)$ and this complete the proof. The proof of part (b) is similar to part (a).
\end{proof}	

Recall that a ring  in which ideals are totally ordered by inclusion is called a chained ring. By Proposition \ref{p22} and Proposition 3.3 of \cite{HL} the proof of the next result is clear.
\begin{cor}
Let $R$ be a Noetherian chained ring and $M$ be the maximal ideal of $R$. Then the following statements are equivalent.\\
a) Every ideal in $R$ is weakly annihilated.\\
b) Every ideal in $R$ is strongly annihilated.\\
c) Every ideal in $R$ is countable strongly annihilated.\\
d) $M$ consists of zerodivisors.

\end{cor}

A countable strongly annihilated ideal need not be an annihilator ideal. See Example \ref{e2}. 

\begin{prop}
A countable generated ideal $I$ of  a ring $R$ is  countable strongly annihilated ideal if and only if it is a annihilator ideal.
\end{prop}
\begin{proof}
$(\Leftarrow)$ It is clear.	

$(\Rightarrow)$	Let $I=(a_1,a_2,\ldots)$. If $b\notin I$, then there is $c_b\in R$ such that $c_ba_n=0$, for any $n$, and $c_bb\neq0$. Suppose that $J$ is the ideal generated by $A=\{c_b:b\notin I\}$. It is obvious that $I\subseteq {\rm Ann}(J)$. Now let $t\in{\rm Ann}(J)$. If $t\notin I$, then there is $c_t\in R$ such that $c_ta_n=0$ and $c_tt\neq0$. Consequently, $c_t\in J$ and this implies that $c_tt=0$, which is not true. Hence, ${\rm Ann}(J)\subseteq I$ and thus ${\rm Ann}(J)=I$.
\end{proof}
From \cite{GM} recall that an ideal $I$ of a ring $R$ is called universally contracted if for any extension ring $T$ of $R$, there is an ideal $J\subseteq T$ with $I=J\cap R$, this is equivalent to the condition that $IT\cap R = I$ for each unitary extension ring $T$ of $R$. 
The proof of the next result is similar to the proof of Proposition 3.3 of \cite{GM}.

\begin{prop}
Let $I$ be an ideal of the ring $R$ and consider the following conditions.\\
a) $I$ is a	countable strongly annihilated ideal in $R$.\\
b) $I[x_1,x_2,\ldots]$ is universally contracted for each countable set $\{x_1,x_2,\ldots\}$ of indeterminates over $R$.\\
c) $I$ is universally contracted.

Then $(a\Rightarrow b \Rightarrow c)$.		
\end{prop}

\section{On a question of Gilmer and McAdam}
From \cite{GM} recall that an ideal $I$ of a ring $R$ is called universally proper if each finite subset of $I$ is annihilated by a nonzero element of $R$. Let $I$ be an ideal of $R$. If ${\rm Ann}(I)\neq (0)$, then $I$ is  universally proper, see  Proposition 3.2 of \cite{GM}. By the same theorem, if a maximal ideal $M$ of $R$ is universally proper then $M$ is contracted from each simple ring extension of $R$. 

In  Question 4 of \cite{GM}, Gilmer and McAdam raised the question: If $I$ is an ideal of $R$ such that $I$ is contracted from each simple ring extension of $R$, does it follow that $I$ is universally contracted? We found a negative to this question. This example is given by A. Azarang.

\begin{exam}
Let $T=\frac{\Bbb Z_2[x,y,u,v]}{(x^2,y^2, ux+vy-1)}$ and $I=(x^2,y^2, ux+vy-1)$. Suppose that $\alpha=I(x)$, $\beta=I(y)$, $\gamma=I(u)$ and $\lambda=I(v)$. Then $T=\Bbb Z_2[\alpha, \beta,\gamma,\lambda]$, where $\alpha^2=\beta^2=0$ and $\alpha\gamma+\beta\lambda=1$. Now assume that $R=\Bbb Z_2[\alpha,\beta]$. Clearly, $T$ is a unitary extension ring of $R$. Note that $M=(\alpha, \beta)$ is a maximal ideal of $R$ and in Example \ref{e22} we show that ${\rm Ann}(M)\neq(0)$. Hence, $M$ is a universally proper ideal of $R$ and so it is contracted from each simple ring extension of $R$. However, one can easily see that $MT=T$. This conclude that  $M$ is not a contracted ideal from $T$ and so we are through. 
\end{exam}

\section{Applications to $C(X)$}

In this section we see that the  countable strongly annihilated ideals in $C(X)$ are closely related to the real maximal ideals. As a matter of fact, every $z^\circ$-ideal is a countable strongly annihilated ideal if and only if it is a real maximal ideal. An ideal $I$  in a ring $R$ is said to be  strongly divisible  if foe every $a_1,a_2,\ldots$ in $I$ there exists $c\in I$ and $b_1,b_2,\ldots$ in $R$ such that $a_i=cb_i$, for $i=1,2,\ldots$. We begin by the following lemma.

\begin{lem}\label{l1}
 Every  countable strongly annihilated ideal in $C(X)$ is strongly divisible.
\end{lem}
\begin{proof}
Suppose that $I$ is a countable strongly annihilated ideal. Since $I$ is a $z$-ideal, it is enough to show that $Z[I]$ is closed under countable intersection, by Lemma 4.1 of \cite{Az}. Hence, suppose that  
$\{Z(f_n): n\in\Bbb N\}$ be a subfamily of $Z[I]$. Since $I$ is a $z$-ideal, then $f_n\in I$, for any $n\in\Bbb N$ and we can also assume that $|f_n|\leq 1$. Define $f=\sum_{n=1}^\infty\frac{|f(x)|}{2^n}$. Then $f\in C(X)$ and ${\rm Ann}(f)=\bigcap_{n=1}^\infty{\rm Ann}(f_n)$. If $f\notin I$, then by hypothesis there is $g\in C(X)$ such that $gf_n=0$, for any $n\in \Bbb N$ and $gf\neq0$ which is a contradiction. 
\end{proof}
The converse of the above lemma is not true. For example, we consider the maximal ideal $M_0$ in $C(\Bbb R)$. Then $M_0$ is a real maximal and hence by Corollary 4.2 of \cite{Az} it is a strongly divisible ideal which is not a $z^\circ$-ideal. Let $I$ be a countable generated $z$-ideal and $Z[I]$ is closed under countable intersection, then it is well-known that $I=(e)$ for an idempotent $e$, hence $I$ is a countable strongly annihilated ideal. 
\begin{prop}
A space $X$ is an almost $P$-space if and only if countable strongly annihilated ideals and strongly divisible $z$-ideals coincide.
\end{prop}
\begin{proof}
$(\Rightarrow)$ Let $I$ be a strongly divisible $z$-ideal and suppose that $f_1,f_2,\ldots$ in $I$ and $g\notin I$. By hypothesis, there exists $h\in I$ and $g_1,g_2,\ldots$ in $C(X)$ such that $f_i=hg_i$, for any $i=1,2,\ldots$. Since $I$ is a $z^\circ$-ideal, we infer that ${\rm Ann}(h)\nsubseteq {\rm Ann}(g)$ and hence there is $k\in {\rm Ann}(h)$ but $k\notin {\rm Ann}(g)$. Therefore $kg\neq 0$ and $kf_i=0$, for any $i=1,2,\ldots$. This shows that $I$ is a countable strongly annihilated ideal. Now using the above lemma we are done. 	

$(\Leftarrow)$ For any $x\in X$, the ideal $M_x$ is a real maximal $z$-ideal and by hypothesis is a countable strongly annihilated ideal. Hence, $M_x$ is a $z^\circ$-ideal and therefore $X$ is an almost $P$-space.  	
\end{proof}

Recall that every minimal prime ideal is a $\lambda-z^\circ$-ideal, where $\lambda\leq|P|$, but 
need not be countable strongly annihilated. See the next example.

\begin{exam}\label{e1}
We consider the ideal $O_\sigma$ in $C(\Sigma)$, where $\Sigma$ is the space of 4M in \cite{GJ}. By 4M.4, $O_\sigma$ is a minimal prime $z$-ideal. We claim that $O_\sigma$ is not  countable strongly annihilated. By 4M.1, there is $f\in C(\Sigma)$ such that $Z(f)=\{\sigma\}$. If $n\in\Bbb N$, then $n\notin Z(f)$ and hence there are $g_n, h_n\in C(\Sigma)$ such that $n\in Z^\circ(g_n)$ and $Z(f)\subseteq Z^\circ(h_n)$ with $Z(g_n)\cap Z(h_n)=\emptyset$, for every $n\in \Bbb N$. Note that $h_n\in O_\sigma$, for every $n$. Suppose that $Z(k)=\bigcap_{n=1}^\infty Z(h_n)$, where $k\in C(\Sigma)$. It is sufficient to show that $k\notin O_\sigma$. Otherwise, $\sigma\in Z^\circ(k)$ implies that there is a subset $\emptyset\neq U\subseteq \Bbb N$, such that $U\cup\{\sigma\}\subseteq\bigcap_{n=1}^\infty Z(h_n)$. Hence, there is a $i_0\in \Bbb N$ such that $i_0\in Z(h_n)$, for any $n$. This conclude that $Z(g_{i_0})\cap Z(h_{i_0})\neq\emptyset$ which is not true. Therefore, by Lemma \ref{l1}, $O_\sigma$ is not a countable strongly annihilated ideal.  
\end{exam}

\begin{prop}\label{p1}
A maximal ideal  $M$ of $C(X)$ is countable strongly annihilated if and only if it is a $z^\circ$-ideal and real maximal ideal.
\end{prop}
\begin{proof}
$(\Rightarrow)$ Since $M$ is a $z$-ideal, then similar to the proof of  Lemma \ref{l1}, $Z[M]$ is closed under countable intersection. Now  by Theorem 5.14 of \cite{GJ}, $M$ is real.

$(\Leftarrow)$ Let $f_n\in M$, for any $n\in\Bbb N$ and $g\notin M$.  Define $f(x)=\sum_{n=1}^\infty\frac{|f(x)|}{2^n}$. Then $f\in C(X)$ and ${\rm Ann}(f)=\bigcap_{n=1}^\infty{\rm Ann}(f_n)$. Since $M$ is real, by Theorem 5.14 of \cite{GJ}, we infer that $f\in M$. Now by hypothesis, there is $h\in C(X)$ such that $hf=0$ and $hg\neq0$. Hence, $hf_n=0$, for any $n\in \Bbb N$ and we are done. 	
\end{proof}
\begin{exam}\label{e2}
We consider the ideal $M_s$ of the space $S$ of 4N in \cite{GJ}. By Proposition \ref{p1}, $M_s$ is countable strongly annihilated. Now let there is a nonzero ideal $I$ in $C(S)$ such that $M_s={\rm Ann}(I)$. Assume that $0\neq f\in I$. Let $s\neq x_\alpha\in S$ an arbitrary element of $S$. We define $f_{\alpha}(x)=0$ if $x\neq x_\alpha$ and $f_{\alpha}(x_\alpha)=1$. Then $f_\alpha\in C(S)$ and $s\in Z(f_\alpha)$ implies that $f_\alpha\in M_s$, for any $\alpha$. Hence $ff_\alpha=0$, for any $\alpha$. But $f_{\alpha}(x_\alpha)=1$ conclude that $f(x_\alpha)=0$, for any $\alpha$. This shows that $S-\{s\}\subseteq Z(f)$ and therefore we have $S-\{s\}=Z(f)$. Hence $coz(f)=\{s\}$, that is, $coz(f)$ is an open set which contains $s$. Therefore by definition of topology on $S$, the zero set  $Z(f)$ must be a countable set, which implies that $S$ is countable and this is a contradiction.
\end{exam}

\begin{cor}\label{c1}
 Let $X$ be a pseudocompact space. Every maximal ideal of $C(X)$ is  countable strongly annihilated if and only if it is a $z^\circ$-ideal. 
\end{cor}
\begin{proof}
It is trivial, by Theorem 5.8 of \cite{GJ} and Proposition \ref{p1}. 	
\end{proof}
The condition of maximality in the above result is necessary. See the next example.

\begin{exam}
Let $X^*=X\cup\{\alpha\}$ be the one-point compactification of $X$, which $X$ is a uncountable discrete space. The ideal $O_\alpha$ is a $sz^\circ$-ideal but not maximal. We define $f_n\in C(X^*)$ such that $Z(f_n)=X^*-\{x_n\}$, for any $x_n\in X$. Note that $x_n\neq x_m$, for $n\neq m$. Clearly, $f_n\in O_\alpha$, for any $n$. Suppose that $Z(f)=\bigcap_{n=1}^\infty Z(f_n)$, then $Z(f)=X^*-\{x_1,x_2,\ldots\}$. If $\alpha\in Z^\circ(f)$, then there is a finite subset $F$ of $X$ such that $\{\alpha\}\cup(X-F)\subseteq Z(f)$. This shows that $\{x_1,x_2,\ldots\}\subseteq F$ which is not true. Therefore $f\notin O_\alpha$ and so $O_\alpha$ is not countable strongly annihilated. 

\end{exam}
\begin{prop}
Let $I$ be a $z$-ideal in $C(X)$ and $M$ be a maximal ideal containing $I$. Then the following statements are equivalent.\\
a) $\frac{M}{I}$  is a	$z^\circ$-ideal in $\frac{C(X)}{I}$.\\
b) $\frac{M}{I}$  is a	nonregular ideal in $\frac{C(X)}{I}$.\\
c) $\frac{M}{I}$  is a	universally proper ideal in  $\frac{C(X)}{I}$.\\
d) $M=\bigcup_{P\in {\rm Min}(I)} P$.
	
\end{prop}
\begin{proof}
$(a\Rightarrow b)$ It is clear.	

$(b\Rightarrow c)$ We use part (3) of Theorem 3.2 in \cite{GM}. Let $I(f_1),\ldots,I(f_n)$ be  a finite subset  of $\frac{M}{I}$. Then by hypothesis $I(f_1^2+\cdots+f_n^2)$ is a zerodivisor of  $\frac{C(X)}{I}$, thus there exists $I(g)\neq I$ such that $I((gf_1)^2+\cdots+(gf_n)^2)=I$. Therefore, $(gf_1)^2+\cdots+(gf_n)^2\in I$ and since $I$ is a $z$-ideal then  $gf_i\in I$, and hence $I(gf_i)=I$, for every $1\leq i\leq n$. This complete the proof.

$(b\Rightarrow c)$ It is clear by part (3) of Theorem 3.2 in \cite{GM}.

$(b\Rightarrow a)$ Since $I$ is  semiprime, then $\frac{C(X)}{I}$ is a reduced ring and it is also satisfies the property $A$, by Lemma 4.1 of \cite{AAHS}. Now Lemma 1.22 of \cite{AKR} shows that $\frac{M}{I}$  is a	$z^\circ$-ideal in $\frac{C(X)}{I}$. 

$(a\Leftrightarrow d)$ It is clear by Theorem 4.2 of \cite{AAHS}.  
\end{proof}

In Lemma 4.1 of \cite{AAHS} it is shown that if $I$ is a semiprime ideal of $C(X)$ then the factor ring $\frac{C(X)}{I}$ satisfies the property $A$. However, in the next example we show that the converse is not true.

\begin{exam} 
Recall that every ideal in  $C(\Sigma)$ is absolutely convex. Now let $f\in C(\Sigma)$ and $I=(f)$ such that $Z(f)=\{\sigma\}$. Since $Z(f)$ is not an open set, then $I$ is not a semiprime ideal. We claim that  $\frac{C(\Sigma)}{I}$  satisfies the property $A$. Assume that $\frac{J}{I}$ be an ideal of $\frac{C(\Sigma)}{I}$ consisting of zerodivisors generated by
finitely many elements $I(f_1),\ldots, I(f_n)$, where $f_1,\ldots f_n\in J$.	Hence, $I(|f_i|)\in \frac{J}{I}$, for every $1\leq i\leq n$ and so there exists $I(|g|)\neq I$ such that $I(|gf_1|+\cdots+|gf_n|)=I$. This implies that $|gf_1|+\cdots+|gf_n|\in I$. Since $|gf_i|\leq |gf_1|+\cdots+|gf_n|$, we infer that $|gf_i|\in I$, for every $1\leq i\leq n$ and so we are through.
\end{exam}

In the next result $q(X)$, denotes the classical ring of quotients $C(X)$.

\begin{prop}[\cite{AAHS}, Corollary 5.5] The following statements are equivalent.\\
a) Every maximal ideal of $q(X)$ is real.\\
b) $X$ is a pseudocompact almost $P$-space.\\
c) The set of all real maximal ideals of $C(X)$ coincides with the set of all maximal $z^\circ$-ideals of $C(X)$.
	
\end{prop}

In the following proposition we see another equivalent condition for pseudocompact almost $P$-spaces in terms of  countable strongly annihilated ideals. 

\begin{prop}\label{p2}
A space $X$ is a pseudocompact almost $P$-space if and only if every maximal ideal of $C(X)$ is countable strongly annihilated.
\end{prop}	
\begin{proof}
$(\Rightarrow)$ Let $M$ be a maximal ideal of $C(X)$. Since $X$ is an almost $P$-space, then $M$ is $z^\circ$-ideal, by Theorem 2.14 of \cite{AKR0}. Also because $X$ is pseudocompact, then by Corollary \ref{c1}, we conclude that  $M$  is  countable strongly annihilated.

$(\Leftarrow)$
By Proposition \ref{p1}, every maximal ideal is a $z^\circ$-ideal and real maximal, hence by   Theorem 2.14 of \cite{AKR0}, $X$ is an almost $P$-space and by  Theorem 5.8 of \cite{GJ}, $X$ is a pseudocompact space. 
\end{proof}	
\begin{Rem}
a) Let $X$ be a realcompact almost $P$-space. Then every maximal ideal of $C(X)$ is fixed if and only if it is a countable strongly annihilated ideal.\\
b) Let $X$ be a compact almost $P$-space. Then every maximal ideal of $C(X)$ is countable strongly annihilated.
\end{Rem}

\begin{Rem}
If every ideal of $C(X)$ is  countable strongly annihilated then $X$ is a $P$-space. The converse of this fact is not true. For example, the space $\Bbb N$ is a $P$-space and since it is not pseudocompact, then by Proposition \ref{p2}, there is a maximal ideal which is not   countable strongly annihilated. 
\end{Rem}

\begin{prop}\label{p33}
The following statements are equivalent for an ideal $I$ of a $P$-space $X$.\\
a) $I$  is a	countable strongly annihilated ideal.\\
b) $Z[I]$ closed under countable intersection.\\
c) $I$ satisfies in c.s.a.c..
	
\end{prop}
\begin{proof}
$(a\Rightarrow b)$ It is clear by Lemma \ref{l1}.

$(b\Rightarrow c)$ Let $S=\{f_n:n=1,\ldots\}$ be a countable subset of $I$. There is $f\in C(X)$ such that  ${\rm Ann}(f)=\bigcap_{n=1}^\infty{\rm Ann}(f_n)$ and $\bigcap_{n=1}^\infty Z(f_n)=Z(f)\in Z[I]$. Since $I$ is a $z$-ideal, then $f\in I$ and we are through. 
 
$(c\Rightarrow a)$ Let $S=\{f_n:n=1,\ldots\}$ be a countable subset of $I$ and $g\notin I$. Suppose that $f\in I$ in which ${\rm Ann}(f)=\bigcap_{n=1}^\infty{\rm Ann}(f_n)$ and $\bigcap_{n=1}^\infty Z(f_n)=Z(f)\in Z[I]$. We define $h(x)=1$ if $x\in Z(f)$ and $h(x)=0$ if $x\in coz(f)$. Then $h$ is a continuous function on $X$. Clearly, $hf=0$ and hence $hf_n=0$, for every $n\in \Bbb N$. It is sufficient to show that $hg\neq 0$. Otherwise, $Z(f)\subseteq coz(h)\subseteq Z(g)$ implies that $Z(f)\subseteq Z^\circ(g)$. By 1D.1 of \cite{GJ} there is $k\in C(X)$ such that $g=fk$ which conclude that  $g\in I$ and a contradiction.
\end{proof}

\begin{cor}
The following statements are equivalent for a $z^\circ$-ideal $I$ of $C(X)$.\\
a) $I$  is a	countable strongly annihilated ideal.\\
b) $Z[I]$ closed under countable intersection.\\
c) $I$ satisfies in c.s.a.c..
	
\end{cor}
\begin{proof}
It is similar to the proof of Proposition \ref{p33}.
\end{proof}

In Theorem 4.4 of \cite{Az} it is shown that $X$ is Lindel\"{o}f if and only if every strongly divisible ideal is fixed. If $X$ is an almost $P$-space, then the same result holds for  countable strongly annihilated ideals, see the next proposition. 
\begin{prop}
An almost $P$-space $X$ is Lindel\"{o}f if and only if every countable strongly annihilated ideal is fixed.	
\end{prop}
\begin{proof}
$(\Rightarrow)$ It is hold without assuming that the space is an almost $P$-space. Suppose that $X$ is Lindel\"{o}f and $I$ is a  countable strongly annihilated ideal. In proving Lemma \ref{l1}, we show that $Z[I]$ has the countable intersection property and hence $Z[I]$ is fixed, i.e., $I$ is fixed.  	
	
$(\Leftarrow)$ It is similar to the proof of Theorem 4.4 of \cite{Az}. Just to point out that since $X$ is an almost $P$-space, every $z$-ideal is a $z^\circ$-ideal and we use the above corollary.	
\end{proof}

It is clear that every real maximal ideal of $C(X)$ satisfy in c.s.a.c.. The proof of the next result is clear by Propositions \ref{p1} and \ref{p33}.

\begin{prop}\label{P1}
Let $R$ be a reduced ring with s.a.c.. For a maximal ideal $M$ of $R$, the following conditions are equivalent.\\
a) $M$ is a universally contracted ideal.\\
b) $M$ is a $sz^\circ$-ideal.\\
c) $M$ is a $z^\circ$-ideal.\\
d) $M$ is a nonregular ideal.\\
e) Each finite subset of $M$ is annihilated by
a nonzero element of $R$.\\
f) $M[\alpha]\neq R[\alpha]$ for each simple ring extension $R[\alpha]$ of $R$.\\
g) $M$ is  a contracted proper ideal.\\
h) $M$  is contracted from each simple ring extension of $R$.\\
i) $M[x]$ is a universally contracted ideal in $R[x]$.\\
j) $M[x]$ is a $sz^\circ$-ideal in $R[x]$.\\
k) $M[x]$ is a $z^\circ$-ideal in $R[x]$.\\
In case $R=C(X)$ and $M$ is a real maximal ideal of $C(X)$, the following statement is equivalent to the above statements.\\
l) $M$	is a countable strongly annihilated ideal.\\
In case $R=C(X)$ and $X$ is a finite space, the following statements are equivalent to the above statements.\\
m) $M[[x]]$	is a countable strongly annihilated in $R[[x]]$.\\
n) $M[[x]]$	is a $z^\circ$-ideal in $R[[x]]$.
	
\end{prop}

\begin{proof}
$(a\Leftrightarrow b)$ It is clear by Theorem 3.5 in \cite{GM}.

$(b\Rightarrow c)$ It is clear by Proposition 1.2 in \cite {GM}.

$(c\Rightarrow d)$ It is clear.

$(d\Rightarrow e)$ It is clear for $R$ satisfy the s.a.c..

$(e\Leftrightarrow f)$ and $(f\Leftrightarrow g)$  They are clear by Proposition 3.2 in \cite{GM}.

$(f\Rightarrow h)$ Let $R[\alpha]$ be a simple ring extension of $R$. Clearly $M[\alpha]$ is an ideal of $R[\alpha]$. Hence $M[\alpha]\cap R$ is an ideal of $R$ which contains $M$. If $M[\alpha]\cap R=R$, then  $R[\alpha]= M[\alpha]$ which is a contradiction. Therefore  $M[\alpha]\cap R=M$ by maximality of $M$ and we are done.

$(h\Rightarrow f)$ On the contrary assume that there exists $\alpha$
such that $M[\alpha]= R[\alpha]$. We claim that $M$ is not contracted from $R[\alpha]$. Suppose that there is an ideal ${\mathcal I}$ of $R[\alpha]$ such that $M={\mathcal I}\cap R$. Clearly, $M[\alpha]= R[\alpha]\subseteq {\mathcal I}$. Hence ${\mathcal I}=R[\alpha]$ which implies that $M=R$ and it is a contradiction.

$(e\Rightarrow d)$ It is clear.

$(d\Rightarrow c)$ It is clear by Corollary 1.22 in \cite{AKR}.

$(c\Rightarrow b)$ It is obvious by Proposition 2.8 in \cite{AM1}.

$(i\Leftrightarrow j)$, $(j\Leftrightarrow k)$ and $(k\Leftrightarrow a)$ See Theorem 3.5 in \cite{GM}. 

$(l\Rightarrow c)$ It is obvious.

$(c\Rightarrow l)$ It is clear by Proposition \ref{p1}.

$(m\Leftrightarrow l)$ In this case $C(X)$ is  a Noetherian ring. Now by Corollary 3.16 and  Proposition 3.19 in \cite{AM2} we are done.

$(n\Leftrightarrow l)$ It is similar to the above item and by Corollary 3.16 and  Proposition 3.20 in \cite{AM2} we are done.
\end{proof}

\begin{cor}
Let $X$ be a compact space. A maximal ideal $M$ of $C(X)$ is $z^\circ$-ideal if and only if it is countable strongly annihilated.
\end{cor}
\begin{proof}
In this case by Theorem 4.11 of \cite{GJ} every maximal ideal in $C(X)$ is of the form $M_p$, for a $p\in X$. Now by implication $(l\Leftrightarrow c)$ of Proposition \ref{P1} we are done. 
\end{proof}

\begin{prop}
The following statements are equivalent.\\
a)  $X$ is an  almost $P$-space.\\
b) Every $M_x$ is an $\aleph_0-z^\circ$-ideal.\\
c) Every $M_x$ is a countable strongly annihilated ideal.

\end{prop}
\begin{proof}

$(a\Rightarrow b)$ Let $S=\{f_n:n=1,\ldots\}$ be a countable subset of $M_x$ and suppose that $f\in M_x$ in which ${\rm Ann}(f)=\bigcap_{n=1}^\infty{\rm Ann}(f_n)$. We claim that $P_S\subseteq P_f$. Let $g\in P_S$ and $f\in Q$, where $Q$ is a minimal prime ideal of $C(X)$. Then there is $h\notin Q$ such that $hf=0$ and hence $hf_n=0$, for every $n\in \Bbb N$. This shows that $f_n\in Q$, for every $n$, and consequence that $g\in Q$. Therefore $g\in P_f$. Since $M_x$ is a $z^\circ$-ideal, we infer that $P_S\subseteq P_f\subseteq M_x$ which complete  the proof. 

$(b\Rightarrow c)$ Let $S=\{f_n:n=1,\ldots\}$ be a countable subset of $M_x$ and  $g\notin M_x$. Suppose that $f\in M_x$ in which ${\rm Ann}(f)=\bigcap_{n=1}^\infty{\rm Ann}(f_n)$. Then $P_f\subseteq M_x$ and hence $g\notin P_f$. Therefore there is a minimal prime ideal $Q$ such that $f\in Q$ but $g\notin Q$. Hence there is $h\notin Q$ such that $hf=0$ and $hg\neq0$ and we are done.

$(c\Rightarrow a)$ It is clear by Theorem 2.14 of \cite{AKR0}.
\end{proof}
\begin{exam}\label{e11}
We consider the space $\Sigma$ of 4M in \cite{GJ}. Let $f\in C(\Sigma)$ with $Z(f)=\{\sigma\}$. Since $f$ is a regular element, by Corollary 2.6 in  \cite{AES}, then every element of $\frac{C(\Sigma)}{I}$ is either a
unit or a zerodivisor, where $I=(f)$. Hence, the maximal ideal $\frac{M_\sigma}{I}$ is a nonregular ideal. One can easily see that every prime ideal containing $f$ is contained in $M_\sigma$. This implies that, for every subset $S$ of $\frac{M_\sigma}{I}$, the intersection of all minimal prime ideals containing $S$ is a subset of $\frac{M_\sigma}{I}$. Since $|M_\sigma|=c$, we conclude that $\frac{M_\sigma}{I}$ is a $c-z^\circ$-ideal. 
\end{exam}

\begin{Rem}
a) 	Recall that $C_F(X)$ is the socle of $C(X)$ and it is well-known  that $C_F(X)=\{f\in C(X):coz(f)~\mbox{is finite}\}$, see \cite{KR}. $C_F(X)$ is always a $z^\circ$-ideal but it is not necessarily a countable strongly annihilated ideal, even if $X$ is a $P$-space. For example, let $f_n\in C(\Bbb N)$ such that $coz(f_n)=\{n\}$, for any $n\in \Bbb N$. Then $f_n\in C_F(X)$, but $\Bbb N=coz(f)=\bigcup_{n=1}^\infty coz(f_n)\notin C_F(X)$.\\
b) $SC_F(X)=\{f\in C(X):coz(f)~\mbox{is countable}\}$ is called the super socle of $C(X)$, the reader refer to \cite{GKN} for details.  $SC_F(X)$ is always a $z$-ideal and it is clear that it closed under countable intersection. If $X$ is a weak $P$-space (a space which every countable set is closed) then $SC_F(X)$ is a $z^\circ$-ideal and so it is a countable strongly annihilated ideal.\\
c) If $SC_F(X)$ is a $z^\circ$-ideal, then it is not necessarily $X$ is a weak $P$-space. For example, let $X^*=X\cup\{\alpha\}$ be a one-point compactification of $X$, which $X$ is a uncountable discrete space. Obviously, $X^*$ is not a weak $P$-space. On the other hand if $A$ is an infinite countable subset of $X^*$, then ${\rm cl}(A)=A\cup\{\alpha\}$ whence we conclude that  $SC_F(X)$ is a $z^\circ$-ideal.\\
d) One can easily see that  $SC_F(X)$ is weakly annihilated if and only if it is a countable strongly annihilated ideal.\\
e) It is easy to see that a discrete space $X$ is finite if and only if $C_F(X)$ is a  countable strongly annihilated ideal.\\
f) Let $X$ be a countably compact space and every $\sigma$-compact subset of $X$ is closed. Then $C_K(X)=\{f\in C(X):\overline{coz(f)}~\mbox{is compact}\}$ is a countable strongly annihilated ideal. To see this first note that $C_K(X)$ is always a $z^\circ$-ideal. Now suppose that  $S=\{f_n:n=1,\ldots\}$ be a countable subset of $C_K(X)$ and let $coz(f)=\bigcup_{n=1}^\infty coz(f_n)$. One can easily see that  $\overline{coz(f)}=\bigcup_{n=1}^\infty\overline{(coz(f_n))}$. Hence, $\overline{coz(f)}$  is $\sigma$-compact and so it is  Lindel\"{o}f. Furthermore, $\overline{coz(f)}$ is also a countably compact. Therefore, it is a compact subset of $X$ and consequence that $f\in C_K(X)$ and we are trough.
\end{Rem}

\vspace{0.5 cm}\noindent {\bf Acknowledgements.} The author is
grateful to Dr. A. Azarang for giving me Example 3.1. Also the author is grateful to the Research Council of Shahid Chamran University of Ahvaz financial support (GN:SCU.MM401.648).

{}
\end{document}